\RequirePackage[l2tabu, orthodox]{nag}
\documentclass[11pt]{article}

\usepackage{amssymb, amsmath, amsthm}
\usepackage{mathtools}
\usepackage{stmaryrd}
\usepackage{lmodern}
\usepackage[stretch=10,shrink=10]{microtype}
\usepackage[hmargin = 3.5cm, vmargin = 2cm, includefoot]{geometry}
\usepackage{xparse}

\usepackage[colorlinks = true, pdfstartview = FitV, linkcolor = blue, citecolor = blue, urlcolor = blue]{hyperref}

\newcommand*{\NN}{\mathbb{N}}
\newcommand*{\ZZ}{\mathbb{Z}}
\newcommand*{\QQ}{\mathbb{Q}}
\newcommand*{\RR}{\mathbb{R}}

\newcommand*{\define}[1]{\textbf{#1}}

\DeclareMathOperator{\im}{im}
\DeclareMathOperator{\supp}{supp}

\DeclarePairedDelimiterX\intervC[2]{[}{]}{#1, #2}
\DeclarePairedDelimiterX\intervO[2]{]}{[}{#1, #2}
\DeclarePairedDelimiterX\intervCO[2]{[}{[}{#1, #2}
\DeclarePairedDelimiterX\intervOC[2]{]}{]}{#1, #2}

\DeclarePairedDelimiterX\intervZ[2]{\llbracket}{\rrbracket}{#1, #2}

\DeclarePairedDelimiter\abs{\lvert}{\rvert}
\DeclarePairedDelimiterX\norm[1]{\lVert}{\rVert}{#1}

\DeclarePairedDelimiterX\group[2]{\langle}{\rangle}{#1 \mid #2}
\DeclarePairedDelimiter\groupe{\langle}{\rangle}

\DeclarePairedDelimiterX\set[2]{\{}{\}}{#1\:\delimsize\vert\:\mathopen{}#2}

\newcommand\colonEquiv{\mathrel{\vcentcolon\Leftrightarrow}}

\theoremstyle{plain}
\newtheorem{theorem}{Theorem}[section]
\newtheorem{proposition}[theorem]{Proposition}

\newtheorem{lemma}[theorem]{Lemma}
\newtheorem{corollary}[theorem]{Corollary}
\theoremstyle{definition}
\newtheorem{definition}[theorem]{Definition}
\newtheorem{notation}[theorem]{Notation}
\newtheorem{example}[theorem]{Example}
\theoremstyle{remark}
\newtheorem{remark}[theorem]{Remark}

\title{On the magnitude homology of metric spaces}
\date{\today}
\author{Beno\^{\i}t Jubin}

\begin{document}

\maketitle

\begin{abstract}
Magnitude homology of enriched categories, and in particular of metric spaces, was recently introduced by T.~Leinster and M.~Shulman.
In this article, we prove that metric spaces satisfying a reasonably mild condition have vanishing magnitude homology groups in nonzero degrees.
\end{abstract}

\section*{Introduction}
\label{sec:intro}

In~\cite{LS}, T.~Leinster and M.~Shulman introduced the \emph{magnitude homology} of certain enriched categories.
This magnitude homology is in particular well-defined for metric spaces, viewed as $\intervCO{0}{+\infty}$-enriched categories.
For a metric space $X$, they completely describe the magnitude homology groups $H_0(X)$ and $H_1(X)$.
They also give two sufficient conditions to ensure $H_2(X) = 0$.
In this article, we prove that one of these conditions, namely being Menger-convex geodetic cut-free (terms defined below), actually ensures that $H_n(X) = 0$ for $n \neq 0$.
This is for instance the case for convex subsets of $\RR^d$ with Euclidean metric and for complete Riemannian manifolds with empty cut-locus.

The article is organized as follows.
After recalling some background material and setting notation in Section~\ref{sec:background}, we define the magnitude homology of metric spaces in Section~\ref{sec:homology} and we study the cases of degrees~0 and~1 in Section~\ref{sec:degrees01}, all of which was already done in~\cite{LS}.
In Section~\ref{sec:simple}, we study the special case of simple chains in the magnitude complex.
We then recall the two important notions introduced in~\cite{LS} of \emph{cut-free} (Section~\ref{sec:cut-free}) and \emph{geodetic} (Section~\ref{sec:geodetic}) spaces, give some characterizations, and prove two important properties (respectively, a decomposition of the magnitude complex, and an ordering of points on segments).
Section~\ref{sec:acyclic} contains the main result of the article: acyclicity of Menger-convex geodetic cut-free spaces.
Finally, we study in Section~\ref{sec:riem} the case of complete Riemannian manifolds.

\paragraph{Acknowledgments}
I would like to thank Michael Shulman, who brought to my attention the recent preprint~\cite{KY}, as well as Masahiko Yoshinaga.
It turns out that both~\cite{KY} and the present article, written independently, prove the same acyclicity result by using essentially the same first step: a direct sum decomposition of the magnitude complex, and differing in the rest of the proof.\footnote{The notions of ``straight'' and ``crooked'' defined here correspond to ``smooth'' and ``singular'' there.}
The article~\cite{KY} goes further in decomposing the magnitude complex, using tensor products, and gives two other applications, while the present article has some more results about geodetic and cut-free spaces, as well as the Riemannian case.

\paragraph{Conventions and notation}
\begin{itemize}
\item
For $m, n \in \RR$, we set $\intervZ{m}{n} \coloneqq \set{i \in \ZZ}{m \leq i \leq n}$.
\item
Unless otherwise specified, $(X, d)$, or $X$ for short, will denote a metric space, and $n$ will denote a nonnegative integer.
\item
Vector spaces are assumed real.
\item
Connected graphs are considered as metric spaces as follows: the points are the vertices, and the distance between two points is the length (number of edges) of a shortest path connecting them.
\item
Riemannian manifolds are assumed connected, and in particular are metric spaces.
\end{itemize}

\section{Background material}
\label{sec:background}

\subsection{Metric spaces}
\label{subsec:metric}

Let $(X, d)$, or $X$ for short, be a metric space.
A \define{geodesic} (resp.\ \define{local geodesic}) in~$X$ is an isometry (resp.\ a local isometry) from an interval of~$\RR$ with the induced metric to~$X$.

\begin{definition}
A metric space is:
\begin{itemize}
\item
a \define{length space} if $d(x, y) = \inf \{ \ell(c) \mid c \colon [0, 1] \to X, c(0) = x, c(1) = y \}$ for any $x, y \in X$, with obvious notation,
\item
\define{geodesic} if any two points can be connected by a geodesic (\textit{i.e.,} there exists a geodesic containing them in its image),
\item
\define{proper} if its closed balls are compact.
\end{itemize}
\end{definition}

A geodesic space is a length space.
A proper space is complete and locally compact.
Conversely, a complete locally compact length space is geodesic and proper (Hopf--Rinow), and all three hypotheses are necessary in order to obtain either conclusion.

\subsection{Finite sequences in metric spaces}
\label{subsec:sequences}
Let $n \in \NN$.
An \define{$n$-sequence} in~$X$ is a function from $\intervZ{0}{n}$ to~$X$.
An $n$-sequence will be written as $x = (x_0, \ldots, x_n)$.
A (nonempty finite) \define{sequence} is an $m$-sequence for some $m \in \NN$.
The set of sequences in~$X$ is denoted by $X^+ \coloneqq \bigcup_{m \in \NN} X^{m+1}$.
We also call an element of a sequence a \define{vertex}.

The \define{length} of sequences is the function
\begin{gather}
\begin{aligned}
\ell \colon X^+ &\longrightarrow \RR_{\geq 0} \\
x &\longmapsto \sum_{i = 1}^m d(x_{i-1}, x_i) \qquad\text{if $x \in X^{m+1}$}.
\end{aligned}
\end{gather}

An $n$-sequence $x$ is \define{non-stuttering} if $x_{i-1} \neq x_i$ for all $i \in \intervZ{1}{n}$.
Let $i, j \in \intervZ{0}{n}$.
An $n$-sequence $x$ is \define{straight from~$i$ to~$j$} if $d(x_i, x_j) = \sum_{k = i+1}^j d(x_{k-1}, x_k)$ and \define{globally straight} if it is straight from~$0$ to~$n$.
Note that a sequence $x$ is globally straight if and only if $\ell(x) = d(x_0, x_n)$.
Let $k \in \intervZ{1}{n-1}$.
An $n$-sequence is \define{straight at~$k$} if it is straight from~$k-1$ to~$k+1$,
and \define{crooked at~$k$} if it is not straight at~$k$.
For convenience, an $n$-sequence will be assumed to be both straight and crooked at~0 and at~$n$.
A sequence is \define{straight} (resp.\ \define{crooked}) if it is so at all its indices.
Obviously, globally straight implies straight.

\begin{remark}\label{rmk:ambig}
The phrase ``$x$ is straight at~$x_i$'' is ambiguous, since the point $x_i$ can appear as a vertex of $x$ at different indices.
\end{remark}

\begin{notation}
We will also write \emph{non-stuttering} sequences using concatenation.
In particular, if a sequence is written using concatenation, this will imply that it is non-stuttering.

If $(x_0, x_1, x_2)$ is straight (resp.\ crooked) at $1$, then we write ``$(x_0, \bar{x}_1, x_2)$'' (resp.\ ``$(x_0, \check{x}_1, x_2)$'') both to express this fact and to denote that sequence.
This defines two complementary ternary relations on $X$ which are symmetric in their first and third variables.
If $(x_0, \bar{x}_1, x_2)$ (resp.\ $x_0 \bar{x}_1 x_2$), then we say that $x_1$ is \define{between} (resp.\ \define{strictly between}) $x_0$ and $x_2$.
Betweenness is a closed relation (\textit{i.e.} its graph is closed in $X^3$).

This notation is adapted to concatenation.
For instance, the expression $x_0 \check{x}_1 \overline{x_2 x_3} x_4$ both denotes the 4-sequence $(x_0, x_1, x_2, x_3, x_4)$ and expresses the fact that it is non-stuttering, crooked at~1, and straight from~1 to~4.
\end{notation}

The following lemma gathers some elementary properties of crooked and straight sequences in metric spaces that we will use throughout.

\begin{lemma}\label{lem:metric}
In a metric space,
\begin{align*}
( x_0 \bar{x}_1 x_2 \text{ and } x_0 \bar{x}_2 x_3 )
&\text{ implies }
x_0 \overline{x_1 x_2} x_3,\\
( x_1 \bar{x}_2 x_3 \text{ and } x_0 \bar{x}_1 x_3 )
&\text{ implies }
x_0 \overline{x_1 x_2} x_3,\\
(x_0, \check{x}_1, \bar{x}_2, x_3)
&\text{ implies }
x_0 \check{x}_1 x_3,\\
(x_0, \bar{x}_1, \check{x}_2, x_3)
&\text{ implies }
x_0 \check{x}_2 x_3.
\end{align*}

Consecutive subsequences of straight (resp.\ crooked) sequences are straight (resp.\ crooked).
Subsequences of globally straight sequences are globally straight.
If the $n$-sequence $x$ is crooked and non-stuttering and $n \geq 1$, then the $(n+1)$-sequences $x x_{n-1}$ and $x_1 x$ are crooked.
\end{lemma}

\begin{proof}
The first statement is a straightforward consequence of the definitions.
The second statement is deduced from the first by reversal, and similarly the fourth from the third.
The third is obtained from the first by contraposition: if $(x_0, \bar{x}_1, x_3)$, then together with the hypothesis $(x_1, \bar{x}_2, x_3)$, it implies by the second statement that $(x_0, \bar{x}_1, x_2)$, which is not the case.

The statements about subsequences are straightforward.
\end{proof}

\begin{remark}
On the other hand, $x_0 \bar{x}_1 \bar{x}_2 x_3$ need not imply $(x_0, \bar{x}_1, x_3)$ nor $(x_0, \bar{x}_2, x_3)$.
This property will be the defining property of \emph{cut-free} metric spaces defined below (Definition~\ref{def:cut-free}).
\end{remark}

\section{Magnitude homology of metric spaces}
\label{sec:homology}

Let $X$ be a metric space.
Let $n \in \NN$.

A \define{simple $n$-chain} in~$X$ is a non-stuttering $n$-sequence in~$X$.
The set of simple $n$-chains is denoted by $X^{\underline{n+1}}$.
An \define{$n$-chain} in~$X$ is an element of the free abelian group generated by the simple $n$-chains in~$X$.
The group of $n$-chains in~$X$ is denoted by
\begin{equation}
C_n(X) \coloneqq \ZZ \: X^{\underline{n+1}}.
\end{equation}

The \define{boundary map} is given by the alternating sum $d_n \coloneqq \sum_{i = 1}^{n-1} (-1)^i d_n^i \colon C_n(X) \to C_{n-1}(X)$ where the face map $d_n^i$ is defined on simple $n$-chains by $d_n^i (x_0 \cdots x_n) \coloneqq x_0 \cdots \hat{x}_i \cdots x_n$ if $x$ is straight at~$i$ and~0 else, or more compactly
\begin{equation}
d_n^i (x_0 \cdots x_n) \coloneqq x_0 \cdots \hat{\bar{x}}_i \cdots x_n,
\end{equation}
and extended by linearity.
It is convenient to set $d_n^0 = d_n^n \coloneqq 0$.
The boundary of a chain is indeed a chain: if $d_n^ix \neq 0$, then the strict betweenness condition implies $x_{i-1} \neq x_{i+1}$.

The \define{magnitude complex} of $X$ is the complex $(C_\bullet(X), d_\bullet)$ and the \define{magnitude homology} of $X$ is the cohomology of the magnitude complex.
One writes as usual the subgroups of $n$-cycles $Z_n(X) \coloneqq \ker d_n$ and $n$-boundaries $B_n(X) \coloneqq \im d_{n+1}$, and the $n^{\text{th}}$-homology group $H_n(X) \coloneqq Z_n(X) / B_n(X)$.

There is a \define{grading} given by the length of simple chains.
The boundary maps preserving the length, this gives a grading of the homology groups.
We will not use this grading in the rest of this article.

\begin{remark}
As explained in~\cite[Lem.~7.1]{LS}, these definitions are actually the translations in the particular case of metric spaces of the general definitions of magnitude homology for enriched categories.
\end{remark}

\section{Degrees 0 and 1 and Menger-convexity}
\label{sec:degrees01}

The following computations of the zeroth and first homology groups constitute Theorems~7.2 and~7.4 respectively of~\cite{LS}.

Since $d_0 = 0$, one has $Z_0(X) = C_0(X)$.
Since $d_1 = 0$, one has $B_0(X) = 0$, so $H_0(X) = Z_0(X)$.
Therefore,
\begin{equation}
H_0(X) = \ZZ \: X,
\end{equation}
the free abelian group generated by the points of $X$.

Since $d_1 = 0$, one has $Z_1(X) = C_1(X)$.
Since $d_2 = -d_2^1$, one has
\begin{equation*}
B_1(X)
= \groupe{d(x_0 x_1 x_2)}
= \groupe{x_0 \hat{\bar{x}}_1 x_2}
= \group{x_0 x_1}{\exists z \; x_0 \bar{z} x_1}.
\end{equation*}
Therefore,
\begin{equation}\label{eq:H_1}
H_1(X) = \group*{x_0 x_1}{\not\exists z \; x_0 \bar{z} x_1}.
\end{equation}
We recall the following classical definition.

\begin{definition}\label{def:Menger}
A metric space is \define{Menger-convex} if strictly between any two distinct points, there exists a third point.
\end{definition}

\begin{remark}
A Menger-convex space with at least two points has infinitely many points, so the only Menger-convex connected graph is the singleton.
A geodesic space is Menger-convex.
Conversely, a complete Menger-convex space is geodesic.

An open subset of a geodesic space is Menger-convex.
A convex subset of a normed vector space is Menger-convex.
A closed subset of a strictly convex normed vector space is Menger-convex if and only if it is convex.
\end{remark}

\begin{proposition}[{\cite[Cor.~7.6]{LS}}]\label{prop:H_1}
A metric space $X$ is Menger-convex if and only if $H_1(X) = 0$.
\end{proposition}

\begin{proof}
This follows directly from the above computation~\eqref{eq:H_1} of $H_1(X)$.
\end{proof}

\section{Crooked chains and the properties $(*_n)$}
\label{sec:simple}

In this section, we treat separately the case of crooked simple chains.

\begin{remark}
The case of an empty (resp.\ singleton) metric space is very particular since such a space does not have simple $n$-chains for $n \geq 1$ (resp.\ $n \geq 2$).
The following results trivially hold for these spaces, even if the proofs generally assume the existence of crooked simple $n$-chains for any $n \in \NN$ (which holds in metric spaces with at least two points).
\end{remark}

\begin{lemma}\label{lem:no-cancel}
 Let $n \in \NN$ and $i, j \in \intervZ{0}{n}$.
 If $x$ is a simple $n$-chain and $i \neq j$, then $\ZZ \: d_n^i x \cap \ZZ \: d_n^j x = \{0\}$.
\end{lemma}

\begin{proof}
Let $x$ be a simple $n$-chain.
Suppose that $i \leq j$ and $d_n^i x = d_n^j x \neq 0$.
Then, $1 \leq i \leq j \leq n-1$ and $x_k = x_{k+1}$ for all $k \in \intervZ{i}{j-1}$.
Since $x$ is non-stuttering, this implies that $\intervZ{i}{j-1} = \varnothing$, so $i = j$.
\end{proof}

\begin{lemma}\label{lem:crooked}
A simple chain is a cycle if and only if it is crooked.
\end{lemma}

\begin{proof}
Let $x$ be a simple $n$-chain.
Let $i \in \intervZ{1}{n-1}$.
If $x$ is crooked at~$i$, then $d_n^i x = 0$.
Therefore, if $x$ is crooked, then $dx = \sum_{i=1}^{n-1} (-1)^i d_n^i x = 0$, so $x$ is a cycle.

Conversely, if $x$ is a 0-sequence or a 1-sequence, then it is crooked, so we suppose $n \geq 2$.
Let $i \in \intervZ{1}{n-1}$.
If $x$ is straight at~$i$, then $d_n^i x \neq 0$ and it is not cancelled by any other $d_n^j x$ by Lemma~\ref{lem:no-cancel}.
\end{proof}

\begin{definition}\label{def:prop*}
We define the following properties of a metric space~$X$.
\begin{itemize}
\item
Property~\define{$(*{*}*)$}:
for any $x_0 \check{x}_1 \check{x}_2 x_3$, there exists $z \in X$ such that $x_0 \check{x}_1 \bar{z} \check{x}_2 x_3$.
\item
Property~\define{$(**)$}:
for any $x_0 \check{x}_1 x_2$, there exists $z \in X$ such that $x_0 \bar{z} \check{x}_1 x_2$.
\item
Property~\define{$(*_n)$}, $n \geq 1$:
for any crooked simple $n$-chain $x$, there exist $z \in X$ and $i \in \intervZ{1}{n}$ such that $x_0 \cdots \check{x}_{i-1} \bar{z} \check{x}_i \cdots x_n$.
\end{itemize}
\end{definition}

Note that property $(*_1)$ is Menger-convexity.

\begin{proposition}\label{prop:crooked}
Let $n \geq 1$.
A space has property  $(*_n)$ if and only if every crooked simple $n$-chain is a boundary.
\end{proposition}

\begin{proof}
Necessity is obvious.
As for sufficiency, let $x$ be a crooked simple $n$-chain.
By hypothesis, it is a boundary, so there exists an almost zero family of integers $(a_y)_{y \in X^{\underline{n+2}}}$ such that $x = d (\sum_y a_y \: y) = \sum_y a_y \sum_{i=1}^{n-1} (-1)^i \: y_0 \cdots \hat{\bar{y}}_i \cdots y_{n+1}$.
Therefore, the sum contains at least one $y$ of the form $x_0 \cdots \check{x}_{i-1} \bar{z} \check{x}_i \cdots x_n$ for some $z \in X$ and $i \in \intervZ{1}{n}$.
\end{proof}

We also introduce a ``discrete analog'' of geodesicy, which is weaker than geodesicy but sufficient for our purposes.

\begin{definition}\label{def:strMenger}
A metric space $X$ is \define{strongly Menger} if there exists $\alpha > 0$ such that for all $x, y \in X$, there exists $z$ between $x$ and $y$ such that $d(x, z) \geq \alpha \, d(x, y)$.
\end{definition}

A geodesic space is strongly Menger (but not conversely, as $\QQ$ with the standard metric shows).
One easily checks that if $X$ is strongly Menger, then for all $x, y \in X$ and all $s, t \in \intervC{0}{1}$ with $s < t$, there exists $z$ between $x$ and $y$ such that $s \, d(x, y) \leq d(x, z) \leq t \, d(x, y)$.

\begin{proposition}\label{prop:prop*}
Let $n \geq 1$.
The following properties are listed in order of decreasing strength.
\begin{enumerate}
\item
property~$(*{*}*)$,
\item
property~$(**)$,
\item
property~$(*_n)$,
\item
Menger-convexity.
\end{enumerate}
Furthermore, strong Menger-convexity implies property~$(**)$.
\end{proposition}

\begin{proof}
(1)$\Rightarrow$(2).
Let $x_0 \check{x}_1 x_2$ be a 2-chain.
Applying property~$(*{*}*)$ to the 3-chain $x_1 \check{x}_0 \check{x}_1 x_2$, we obtain $z \in X$ such that $x_1 \check{x}_0 \bar{z} \check{x}_1 x_2$.
In particular, one has $x_0 \bar{y} \check{x}_1 x_2$.

(2)$\Rightarrow$(3).
The case $n=1$ follows from (2)$\Rightarrow(3_{n=2})\Rightarrow$(4).
Let $x$ be a crooked simple $n$-chain with $n \geq 2$.
Applying property~$(**)$ to the 2-chain $x_0 \check{x}_1 x_2$, we obtain $z \in X$ such that $x_0 \bar{z} \check{x}_1 x_2$.
Therefore, $x_0 \bar{z} \check{x}_1 \cdots x_n$.

(3)$\Rightarrow$(4).
Let $x_0, x_1 \in X$ with $x_0 \neq x_1$.
Let $y$ be the $n$-chain defined by $y_{2i} \coloneqq x_0$ and $y_{2i+1} \coloneqq x_1$ for $i \in \intervZ{0}{n/2}$.
Then, $y$ is crooked, so property~$(*_n)$ gives the existence of $z \in X$ and $i \in \intervZ{1}{n}$ such that $y_0 \cdots \check{y}_{i-1} \bar{z} \check{y}_i \cdots y_n$.
Whatever the value of $i$, this yields $x_0 \bar{z} x_1$.

Let $X$ be strongly Menger-convex and let $x_0 \check{x}_1 x_2$ be a crooked simple 2-chain in $X$.
The following argument is similar to an argument in the proof of~\cite[Thm.~7.19]{LS}.
By the remark following the definition of strong Menger-convexity, for any $n \in \NN_{>0}$, there exists $z_n$ between $x_0$ and $x_1$ such that $d(x_0, z_n) \leq d(x_0, x_1)/n$.
Since $(z_n)$ converges to $z$ and betweenness is a closed relation, if $z \bar{x}_1 x_2$ for all $n$, then $x_0 \bar{x}_1 x_2$, which is not.
Therefore, there exists $N \in \NN_{>0}$ such that $x_0 \bar{z}_N \check{x}_1 x_2$.
\end{proof}

\begin{remark}
The Riemannian circle of length $2\pi$ is geodesic without property~$(*{*}*)$, as the crooked simple 3-chain $(0, t, 2t, 3t)$ with $\pi/2 < t < 2 \pi/3$ shows.
The set of rational numbers with the standard metric satisfies property~$(*{*}*)$ but is not a length space.
Complete Menger-convex metric spaces are geodesic, so have property~$(**)$.
The Riemannian circle is such an example.
\end{remark}

\begin{corollary}
If $H_n(X) = 0$ for some $n \geq 1$, then $X$ is Menger-convex.
\end{corollary}

\begin{proof}
If $H_n(X) = 0$ for some $n \geq 1$, then all crooked simple $n$-chains, which are cycles by Lemma~\ref{lem:crooked}, are boundaries.
Therefore, $X$ has property~$(*_n)$ by Proposition~\ref{prop:crooked}.
By Proposition~\ref{prop:prop*}, this implies that $X$ is Menger-convex.
\end{proof}

\section{Cut-free spaces and decomposition of the magnitude complex}
\label{sec:cut-free}

In this section, we first recall the important notion of cut-freeness introduced in~\cite{LS} under the name ``with no 4-cut''.
We then prove the important fact that in a cut-free space, there is a natural decomposition of the magnitude complex.

\subsection{Cut-free spaces}

\begin{definition}\label{def:cut-free}
A metric space is \define{cut-free} if $x_0 \bar{x}_1 \bar{x}_2 x_3$ implies $x_0 \bar{x}_2 x_3$.
\end{definition}

The term ``cut-free'' should be understood as ``with no (nontrivial) shortcut''.
The defining property of cut-freeness can also be interpreted as follows: removing a vertex where a sequence is straight does not alter straightness at other vertices.
Compare Lemma~\ref{lem:metric}, which implies that removing such a vertex does not alter crookedness at other vertices.
Therefore, in a cut-free space, removing such a vertex (which is what a boundary map $d_n^j$ does, when nonzero) does not alter straight/crookedness at other vertices.
This will be crucial in the proof of Proposition~\ref{prop:decomp}.

The following proposition shows that cut-freeness implies an apparently stronger property.

\begin{proposition}\label{prop:glob-straight}
A metric space is cut-free if and only if all straight simple chains are globally straight.
In a cut-free space, every local geodesic is a geodesic.
A geodesic space where every local geodesic is a geodesic is cut-free.
\end{proposition}

\begin{proof}
For the first claim, sufficiency is obvious.
We prove necessity.
There is nothing to prove for simple $n$-chains with $n \leq 2$.
If $x_0 \bar{x}_1 \bar{x}_2 x_3$, then $x_0 \bar{x}_2 x_3$, and these two conditions imply $x_0 \overline{x_1 x_2} x_3$ by Lemma~\ref{lem:metric}.
Let $n \geq 4$ and proceed by induction on $n$.
If $x_0 \bar{x}_1 \cdots \bar{x}_{n-1} x_n$, then we apply the case $n = 3$ to $x_0 \bar{x}_1 \bar{x}_2 x_3$ to obtain $x_0 \overline{x_1 x_2} x_3$.
In particular, $x_0 \bar{x}_2 x_3$, so $x_0 \bar{x}_2 \cdots \bar{x}_{n-1} x_n$, and by the induction hypothesis, $x_0 \overline{x_2 \cdots x_{n-1}} x_n$.
Since $x_0 \bar{x}_1 x_2$, one has $x_0 \overline{x_1 \cdots x_{n-1}} x_n$.

The second claim is a direct consequence of the first.
For the third claim, let $x_0 \bar{x}_1 \bar{x}_2 x_3$.
For $i \in \{ 0, 1, 2 \}$, there is a geodesic $\gamma_i$ from $x_i$ to $x_{i+1}$.
Since $x_0 \bar{x}_1 x_2$, the concatenation $\gamma_{01}$ of $\gamma_0$ and $\gamma_1$ is a local geodesic, so a geodesic, and similarly for the concatenation $\gamma_{12}$, as a consequence of $x_1 \bar{x}_2 x_3$.
Therefore, $\gamma_{012}$ is a geodesic, and $x_0 \overline{x_1 x_2} x_3$.
\end{proof}

We recall the following standard definition of graph theory: a \define{hole} in a graph is a cycle of length at least 4 with no chord (every two non-consecutive vertices of the cycle are non-adjacent; equivalently, if forms an induced, or full, subgraph).
Since cut-freeness is hereditary, induced (i.e., full) subgraphs of cut-free graphs are cut-free.

\begin{proposition}
A complete graph is cut-free.
A hole-free connected graph with no cycle of length at least~5 is cut-free
A cut-free connected graph has no hole.
\end{proposition}

In particular, a tree is cut-free, as proved in~\cite[Exa.~7.18]{LS}.

\begin{proof}
The first claim is obvious.

For the second claim, let $x_0 \bar{x}_1 \bar{x}_2 x_3$ be a chain.
For $i \in \{ 1, 2, 3 \}$, let $c_i$ be a shortest path from $x_{i-1}$ to $x_i$.
The only common vertex between $c_1$ and $c_2$ is $x_1$, because $c_1$ and $c_2$ are shortest paths and $x_0 \bar{x}_1 x_2$.
Similarly, the only common vertex between $c_2$ and $c_3$ is $x_2$.
If $c_1$ and $c_3$ have a common vertex $u$, then ($x_0 \bar{u} x_1$ and $x_0 \bar{x}_1 x_2$) implies $u \bar{x}_1 x_2$, and 
 ($x_2 \bar{u} x_3$ and $x_1 \bar{x}_2 x_3$) implies $x_1 \bar{x}_2 u$.
But one cannot have both $x_1 \bar{x}_2 u$ and $u \bar{x}_1 x_2$.
Therefore, $c_1$ and $c_3$ have no vertex in common.
Therefore, the concatenation $c_1 c_2 c_3$ is a shortest path, so $x_0 \overline{x_1 x_2} x_3$.

Let $x'_0$ be the last common vertex of $c_1$ and $c$.
It cannot be $x_1$ since $x_0 \check{x}_1 x_3$.
Let $x'_3$ be the first common vertex of $c_3$ and $c$.
It cannot be $x_2$ since $x_0 \check{x}_2 x_3$.
Then,  $x'_0 \bar{x}_1 \bar{x}_2 x'_3$ and $x'_0 \check{x}_1 x'_3$ and the corresponding restrictions of the $c_i$'s and $c$ are as above.
Therefore, we can suppose that the only common vertex of $c$ and $c_1$ (resp.\ $c_2$) is $x_1$ (resp.\ $x_2$).

Therefore, the concatenation $c_1 c_2 c_3 c$ (with $c$ reversed) is a cycle, and since it contains the distinct vertices $x_0, x_1, x_2, x_3$, it has length at least~4.
If it has length exactly~4, then it is a 4-hole since $(x_0, x_2)$ and $(x_1, x_3)$ are non-adjacent.

For the third claim, let $(x_0, \ldots, x_{n-1})$, with $n \geq 4$, be a hole of minimal length.
Then, $x_0 \bar{x}_1 \bar{x}_2 x_{\lfloor n/2 \rfloor + 1}$, but $x_0 \check{x}_1 x_{\lfloor n/2 \rfloor + 1}$.
\end{proof}

\begin{remark}
The converses do not hold: the tree with three vertices is cut-free but not complete; the complete graph on five vertices has a five-cycle and is cut-free; the cyclic graph with five vertices where two edges are added so that one of its vertices is adjacent to all others, has no hole but is not cut-free.
\end{remark}

\begin{proposition}\label{prop:prop***}
A Menger-convex cut-free space has property~$(*{*}*)$.
\end{proposition}

\begin{proof}
Let $x_0 \check{x}_1 \check{x}_2 x_3$ be a crooked 3-chain.
Since $X$ is Menger-convex, there exists $z \in X$ such that $x_1 \bar{z} x_2$.
Since $x_1 \bar{z} x_2$ and $X$ is cut-free, then $x_0 \check{x}_1 x_2$ implies $x_0 \check{x}_1 z$, and $x_1 \check{x}_2 x_3$ implies $z \check{x}_2 x_3$.
Putting together these three properties, one obtains $x_0 \check{x}_1 \bar{z} \check{x}_2 x_3$.
\end{proof}

\subsection{Decomposition of the magnitude complex}

Let $n \in \NN$.
For any $k \in \intervZ{0}{n}$ and $u \in X^{k+1}$, set
\begin{equation}
C_n(X; u) \coloneqq
\ZZ \: \set{x_{i_0} \overline{\cdots} \check{x}_{i_1} \overline{\cdots} \: \check{\cdot}  \: \ldots \: \check{\cdot} \: \overline{\cdots} \check{x}_{i_{k-1}} \overline{\cdots} x_{i_k} \in X^{\underline{n+1}}}{\forall m \in \intervZ{0}{k} \; x_{i_m} = u_m}
\end{equation}
where it is implied that $0 = i_0 < \cdots < i_k = n$.
In other words, $C_n(X; u)$ is the subgroup of $C_n(X)$ generated by simple $n$-chains which are crooked at some indices $0 = i_0 < \cdots < i_k = n$ and straight at all other indices, and with $x_{i_m} = u_m$ for $m \in \intervZ{0}{k}$ (the sequence $u$ may be stuttering).
If $x \in C_n(X; u)$ is simple, then $0 = i_0 < \cdots < i_k = n$ is the sequence of indices where $x$ is crooked.
If $k > n$, we set $C_n(X; u) \coloneqq 0$.

If $k = 0 < n$, then $C_n(X; u) = 0$.
If $k = 1$ (resp.\ $n$), then the elements of $C_n(X; u)$ are straight (resp.\ crooked).
If $k = n$ and $u$ is non-stuttering, then $C_n(X; u) = \ZZ \: u$.

Since each simple $n$-chain is in a unique $C_n(X; u)$, one has
\begin{equation}\label{eq:decomp}
C_n(X) = \bigoplus_{\substack{u \in X^+}} C_n(X; u).
\end{equation}
Define $Z_n(X; u) \coloneqq Z_n(X) \cap C_n(X; u)$ and $B_n(X; u) \coloneqq B_n(X) \cap C_n(X; u)$.

\begin{proposition}\label{prop:decomp}
Let $X$ be a cut-free metric space.
If $u \in X^+$, then
\begin{equation}
d \big( C_n(X; u) \big) \subseteq C_{n-1}(X; u).
\end{equation}
In particular,
\begin{equation}
Z_n(X) = \bigoplus_{u \in X^+} Z_n(X; u)
\end{equation}
and
\begin{equation}
B_n(X) = \bigoplus_{u \in X^+} B_n(X; u).
\end{equation}
If $C_n(X; u) \neq 0$, then $u$ is a crooked simple $k$-chain.
\end{proposition}

\begin{proof}
Let $x \in C_n(X; u)$ be simple.
Let $i, j \in \intervZ{1}{n-1}$.
If $x$ is crooked at~$j$, then $d_n^j x = 0$.
Suppose that $x$ is straight at~$j$.
Set $k \coloneqq i$ if $i < j$ and $i-1$ if $j < i$.
If $x$ is crooked (resp.\ straight) at $i$, then $d_n^j x$ is crooked (resp.\ straight) at $k$ by Lemma~\ref{lem:metric} (resp.\ by cut-freeness), and $(d_n^j x)_k = x_i$.
Therefore, each $d_n^j x$, hence also $dx$, is in $C_{n-1}(X; u)$.

Let $a \in Z_n(X)$.
By~\eqref{eq:decomp}, we can write $a = \sum_{u \in X^{+}} a_u$ with $a_u \in C_n(X; u)$.
Then, $da = \sum_{u \in X^+} da_u = 0$.
By the first claim, one has $da_u \in C_{n-1}(X; u)$.
Since these groups are in direct sum, this implies $da_u = 0$ for all $u \in X^+$.
Therefore, $a \in \bigoplus_{u \in X^+} Z_n(X; u)$.

The subgroup $B_n(X)$ is generated by the boundaries $dx$ with $x$ a simple $(n+1)$-chain.
Any such $x$ is in some $C_{n+1}(X; u)$, so $dx \in B_n(X; u)$ by the first claim.

Finally, if $C_n(X; u) \neq 0$, then by cut-freeness, $u$ is non-stuttering.
Therefore, by Lemma~\ref{lem:metric}, it is crooked.
\end{proof}

\section{Geodetic spaces and orderings on segments}
\label{sec:geodetic}

In this section, we first recall the important notion of geodeticy introduced in~\cite{LS}.
We then introduce a partial order on the set of points between two given points, which is a total order in geodetic spaces.

\begin{definition}\label{def:geodetic}
A metric space $X$ is \define{geodetic} if $x_0 \bar{x}_1 x_2$ and $x_0 \bar{x}'_1 x_2$ and $x_1 \neq x'_1$ implies $x_0 \bar{x}_1 x'_1$ or $x_0 \bar{x}'_1 x_1$.\end{definition}

\begin{remark}
The conclusion in the definition of geodeticy implies  $x_0 \overline{x_1 x'_1} x_2$ or $x_0 \overline{x'_1 x_1} x_2$.
\end{remark}

An example of a geodetic non-cut-free graph is the cyclic graph of order 5.
An example of a non-geodetic cut-free graph is the complete graph of order 4 with one edge removed.
As noted in~\cite{LS}, geodeticy and cut-freeness are hereditary properties (they are inherited by subspaces).

The following proposition shows that geodeticy is a ``discrete analog'' of unique geodesicy.

\begin{proposition}\label{prop:geodetic}
In a geodetic length space, there exists at most one geodesic connecting any two points.
In particular, a geodesic space is geodetic if and only if it is uniquely geodesic.
\end{proposition}

\begin{proof}
Let $X$ be a geodetic length space.
If $c_1$ and $c_2$ are two geodesics connecting $x$ to $y$, then for all $t \in \intervO{0}{d(x,y)}$, one has $x \overline{c_1(t)}y$ and $x \overline{c_2(t)}y$ and $d(x, c_1(t)) = t = d(x, c_2(t))$, so by geodeticy, $c_1(t) = c_2(t)$.
\end{proof}

We now study the case of normed (real) vector spaces.

\begin{proposition}
A normed real vector space is geodesic.
It is geodetic (equivalently, uniquely geodesic) if and only if it is strictly convex.
It is cut-free if and only if its unit sphere contains no segments $[x, y]$ and $[y, z]$ such that the segment $[x, z]$ is not included in the unit sphere.
In particular, a geodetic normed vector space is cut-free.
\end{proposition}

\begin{proof}
In a normed vector space, straight lines are local geodesics.
That they are the only ones is easily seen to be equivalent to the strict convexity of the unit ball.
Since straight lines are geodesics, geodetic vector spaces are cut-free.

If there are segments as in the proposition, then $(0, \bar{x}, \overline{x + y}, x + y + z)$ but $\norm{x + y + z} \leq \norm{x + z} +\norm{y} < 3$, so the space is not cut-free.
Conversely, suppose there is a cut.
Up to translation, we can suppose that it is of the form $(0, \bar{x}, \overline{x + y}, x + y + z)$ with nonzero $x, y, z$.
Write $x = \norm{x} x_0$ and $y = \norm{y} y_0$ and $z = \norm{z} z_0$.
Then $(0, \bar{x}, \overline{x + y})$  implies that $[x_0, y_0]$ is included in the unit sphere, which for a similar reason also contains $[y_0, z_0]$.
Set $z' \coloneqq \frac{y + z}{\norm{y} + \norm{z}}$.
Since $(0, \check{x}, \overline{x + y + z})$, the segment $[x_0, z']$ is not included in the unit sphere.
Therefore, the triple $(x_0, y_0, z')$ satisfies the conditions of the propoposition.
\end{proof}

\begin{remark}
The normed vector space $\RR^2$ whose unit ball is $\{ (x, y) \in \RR^2 \mid x^2 + y^2 \leq 1 \text{ and } \abs{x} \leq 1/2 \}$ is not strictly convex but is cut-free.
This gives an example of a proper geodesic space which is cut-free but not geodetic.

The proposition also shows that $\RR^d$ with the norm $\ell_1$ or the norm $\ell_\infty$ is not cut-free.
\end{remark}

\begin{example}
We give an example of a proper geodesic geodetic space which is not cut-free.
Consider the 2-dimensional torus $T \coloneqq \{ ( (2 + \cos \theta) \cos \phi, (2 + \cos \theta) \sin \phi, \sin \theta) \in \RR^3 \mid \phi, \theta \in \intervCO{0}{2\pi} \}$ as a Riemannian submanifold of $\RR^3$.
Define the local geodesic $\gamma \colon \RR \to T, t \mapsto (3 \cos t, 3 \sin t, 0)$.
Let $t_0 \in \intervO{0}{\pi}$ be such that $\gamma(t_0)$ is the cut-point (and first conjugate point) of $\gamma(0)$ along $\gamma$.
Let $t_1 \in \intervO{t_0}{\pi}$.
There are exactly two minimizing geodesics from $\gamma(0)$ to $\gamma(t_1)$.
Let $X$ be the simply connected closed subset of $T$ bounded by $\gamma$ and one of these two geodesics.
Then, $X$ is proper geodesic geodetic but is not cut-free.
\end{example}

\begin{proposition}\label{prop:order}
Let $X$ be a metric space.
Let $x_0, x_1 \in X$.
The relation $\preceq_{x_0, x_1}$ on the set of points between $x_0$ and $x_1$ defined by
\begin{equation}
z \preceq_{x_0, x_1} z' \colonEquiv (x_0, \bar{z}, z')
\end{equation}
is a partial order with least (resp.\ greatest) element $x_0$ (resp.\ $x_1$).
If $X$ is geodetic, then this is a total order.
\end{proposition}

\begin{proof}
The relation $\preceq_{x_0, x_1}$ is clearly reflexive.
It is transitive and antisymmetric by Lemma~\ref{lem:metric}.
Totality is an immediate consequence of the geodeticy of~$X$.
\end{proof}

\section{Acyclicity of Menger-convex geodetic cut-free spaces}
\label{sec:acyclic}

Let $n \in \NN$.
Let $a$ be an $n$-chain.
It is a finite sum $a = \sum_{x \in X^{\underline{n+1}}} a_x \: x$ with $a_x \in \ZZ$ for $x \in X^{\underline{n+1}}$ and $(a_x)$ an almost zero family.
We will also write $a(x)$ in place of $a_x$ for the sake of readability.

One has
\begin{align*}
da
&= \sum_{x \in X^{\underline{n+1}}} a_x \sum_{i=1}^{n-1} (-1)^i  x_0 \cdots \hat{\bar{x}}_i \cdots x_n \\
&= \sum_{y \in X^{\underline{n}}} \sum_{i=1}^{n-1} (-1)^i \sum_{z \in X} a(y_0 \cdots y_{i-1} \bar{z} y_i \cdots y_{n-1}) \: y.
\end{align*}
Therefore, $da = 0$ if and only if for all $y \in X^{\underline{n}}$ one has
\begin{equation}\label{eq:cycle}
\sum_{i=1}^{n-1} (-1)^i \sum_{z \in X} a(y_0 \cdots y_{i-1} \bar{z} y_i \cdots y_{n-1}) = 0.
\end{equation}
In other words,
\begin{multline}\label{eq:Z_n}
Z_n(X) =\\
\set[\Big]{\sum_{x \in X^{\underline{n+1}}} a_x \: x \in C_n(X)}{\forall y \in X^{\underline{n}} \; \sum_{i=1}^{n-1} (-1)^i \sum_{z \in X} a(y_0 \cdots y_{i-1} \bar{z} y_i \cdots y_{n-1}) = 0}.
\end{multline}

\begin{lemma}\label{lem:main}
Let $X$ be Menger-convex geodetic cut-free.
If $n \geq 1$ and $u \in X^+$, then $Z_n(X; u) = B_n(X; u)$.
\end{lemma}

\begin{proof}
Let $n \geq 1$ and $u \in X^{\underline{k+1}}$.
If $k= 0$ or $n < k$, then $Z_n(X; u) = 0$.
Therefore, we suppose that $1 \leq k \leq n$.
Let $a \in Z_n(X; u)$.
It is a finite sum $a = \sum_{x \in X^{\underline{n+1}}} a_x \: x$ with $a_x \in \ZZ$ for $x \in X^{\underline{n+1}}$ and $(a_x)$ an almost zero family.
Define the finite set $S \coloneqq \bigcup_{x \in \supp a} \bigcup_{i = 0}^n x_i$.

Let $s_0$ be the $\preceq_{u_0, u_1}$-smallest element of $S$ distinct from $u_0$.
By Menger-convexity of $X$, there exists $r \in X$ such that $u_0 \bar{r} s_0$.
Set
\begin{equation*}
\tilde{a} \coloneqq \sum_x a_x \: x_0 r x_1 \cdots x_n.
\end{equation*}

For all nonzero summands in the above expression, one has $x_0 \bar{r} x_1$ since $(x_0 \bar{x}_1, u_1)$ and $s_0$ is $\preceq_{u_0, u_1}$-minimal (here, we use the fact that $\preceq_{u_0, u_1}$ is a total order, as per Proposition~\ref{prop:order}).
Therefore,
\begin{equation*}
a + d\tilde{a} = \sum_x a_x \sum_{i=1}^{n-1} (-1)^{i+1} \: x_0 r x_1 \cdots \check{\bar{x}}_i \cdots x_n.
\end{equation*}
The coefficient of $y_0 r y_1 \cdots y_{n-1}$ in this sum is $\sum_{i=1}^{n-1} (-1)^{i+1} \sum_z a(y_0 \cdots y_{i-1} \bar{z} y_i \cdots y_{n-1})$, which vanishes by~\eqref{eq:cycle}.
Therefore, $a = - d\tilde{a} \in B_n(X) \cap C_n(X;u) = B_n(X; u)$.
\end{proof}

\begin{theorem}\label{thm:main}
If $X$ is a Menger-convex geodetic cut-free space, then $H_n(X) = 0$ for $n \neq 0$.
\end{theorem}

\begin{proof}
This is a consequence of Lemma~\ref{lem:main} and Proposition~\ref{prop:decomp}.
\end{proof}

\begin{corollary}
If $X$ is a convex subset of a strictly convex normed vector space, then $H_n(X) = 0$ for $n \neq 0$.
\end{corollary}

For the sake of completeness, we recall the following result from~\cite{LS}.

\begin{theorem}[{\cite[Thm.~7.19]{LS}}]\label{thm:LS}
If $X$ is geodetic and has property~$(*_2)$, then $H_2(X) = 0$.
In particular, if~$X$ is uniquely geodesic, then $H_2(X) = 0$.
\end{theorem}

\begin{proof}
Although the statement is slightly more general than that of~\cite[Thm.~7.19]{LS}, the proof there actually proves it.
\end{proof}

\section{Riemannian manifolds}
\label{sec:riem}

We keep the terminology of the preceding sections, so a ``local geodesic'' (resp.\ ``geodesic'') denotes what in Riemannian geometry is generally called a ``geodesic'' (resp.\ ``minimizing geodesic'').
Complete Riemannian manifolds are geodesic (Hopf--Rinow).

Let $M$ be a complete Riemannian manifold, $x \in M$, and $\gamma \colon \RR \to M$ a local geodesic with $\gamma(0) = x$.
If $t \coloneqq \sup \{ s \geq 0 \mid \gamma \text{ is minimizing between $x$ and $\gamma(s)$} \}$ is finite, then one says that $\gamma(t)$ is the \define{cut-point} of $x$ along $\gamma$.
The \define{cut-locus} of $x$ is the set of cut-points of $x$ along local geodesics through $x$, and the cut-locus of $M$ is the set of pairs formed by a point and one of its cut-points.
For details, we refer to~\cite[Ch.~XIII.2]{DC} and~\cite[Ch.~2.1]{K}.

For complete Riemannian manifolds, several of the conditions introduced in~\cite{LS} and this article turn out to be equivalent.

\begin{theorem}\label{thm:riem}
In a complete Riemannian manifold $M$, the following properties are equivalent:
\begin{enumerate}
\item
$M$ has empty cut-locus,
\item
$M$ is cut-free,
\item
$M$ has property~$(*{*}*)$,
\item
$M$ is geodetic.
\end{enumerate}
\end{theorem}

\begin{proof}
(1)$\Rightarrow$(2).
A complete Riemannian manifold has empty cut-locus if and only if all its local geodesics are geodesics.
Since complete Riemannian manifolds are geodesic (Hopf--Rinow theorem), the conclusion follows from Proposition~\ref{prop:glob-straight}.

(2)$\Rightarrow$(3) follows from Proposition~\ref{prop:prop***}.

(3)$\Rightarrow$(1).
We prove the contrapositive.
If~$M$ has a non-empty cut-locus, then one can find along a geodesic $\gamma$ the following configuration in that order:
$x_0$, cut-point of $x_2$, $x_1$, cut-point of $x_3$, cut-point of $x_0$, $x_2$, cut-point of $x_1$, $x_3$.
Then, $x_0 \check{x}_1 \check{x}_2 x_3$ and $\gamma$ is uniquely minimizing on $[x_1, x_2]$.
Therefore, if $x_1 \bar{y} x_2$, then $y$ is on $\gamma$, but cannot be both after the cut-point of $x_0$ and before that of $x_3$.
Therefore, $x_0 \bar{y} x_2$ or $x_1 \bar{y} x_3$.

(1)$\Rightarrow$(4).
In a complete Riemannian manifold with empty cut-locus, each pair of points is connected by a unique geodesic, see for instance~\cite[Cor.~XIII.2.8]{DC}.

(4)$\Rightarrow$(1).
The cut-locus of a point is the closure of the set of points that can be connected to it by two distinct minimizing geodesics (see~\cite[Thm.~2.1.14]{K}).
If~$M$ is geodetic, the latter set is empty, and so is its closure.
\end{proof}

\begin{corollary}
If $M$ is a complete Riemannian manifold with empty cut-locus, then $H_n(M) = 0$ for $n \neq 0$.
\end{corollary}

\vspace*{5mm}
\noindent
\parbox[t]{21em}
{\scriptsize
Beno{\^i}t Jubin\\
Sorbonne Universit{\'e}s, UPMC Univ Paris 6\\
Institut de Math{\'e}matiques de Jussieu\\
F-75005 Paris France\\
e-mail: benoit.jubin@imj-prg.fr
}
\end{document}